\documentclass{aims} 
\usepackage{amsmath, amsfonts}
\usepackage{mathtools}
\usepackage{xcolor}
\usepackage{paralist}
\usepackage[misc]{ifsym}
\usepackage{epsfig} 
\usepackage{epstopdf} 
\usepackage[colorlinks=true]{hyperref}
\usepackage{listings}
\usepackage{courier}
\usepackage{layouts}
 
\definecolor{codegreen}{rgb}{0,0.4,0}
\definecolor{codegray}{rgb}{0.5,0.5,0.5}
\definecolor{codepurple}{rgb}{0.58,0,0.12}
\definecolor{backcolour}{rgb}{1.0,0.99,0.95}
\definecolor{deepblue}{rgb}{0.,0.1,0.8}
\definecolor{deepred}{rgb}{0.6,0,0}
\definecolor{deepgreen}{rgb}{0,0.5,0}
 
\lstdefinestyle{mystyle}{
    backgroundcolor=\color{backcolour},   
    commentstyle=\color{codegreen},
    keywordstyle=\color{deepblue},
    numberstyle=\tiny\color{codegray},
    stringstyle=\color{codepurple},
    breakatwhitespace=false,         
    breaklines=true,                 
    captionpos=b,                    
    keepspaces=true,                 
    showspaces=false,                
    showstringspaces=false,
    showtabs=false,                  
    tabsize=2,
    aboveskip=1em,
    belowskip=1em
}
 
\hypersetup{urlcolor=blue, citecolor=red}
\allowdisplaybreaks

\def\currentvolume{X}
 \def\currentissue{X}
  \def\currentyear{200X}
   \def\currentmonth{XX}
    \def\ppages{X-XX}
     \def\DOI{10.3934/jcd.2026001}

\newtheorem{theorem}{Theorem}[section]
\newtheorem{corollary}[theorem]{Corollary}

\newtheorem{lemma}[theorem]{Lemma}
\newtheorem{proposition}[theorem]{Proposition}

\theoremstyle{definition}

\newtheorem{remark}[theorem]{Remark}

\newtheorem*{openproblem}{Open Problem}

\newcommand{\code}[1]{\lstinline$#1$}
\newcommand{\quflow}{\texttt{\small quflow}}
\newcommand{\edit}[1]{#1}

\title[Introduction to matrix hydrodynamics]
{A brief introduction to matrix hydrodynamics} 

\author[Klas Modin and Milo Viviani]{}


\thanks{The authors would like to thank all the participants of the mini-workshop  ``Geometric and Stochastic Methods for Fluid Models'', Kristineberg, Sweden, October 2024.}

\thanks{$^*$Corresponding author: Klas Modin}

\begin{document}
\maketitle

\centerline{\scshape
Klas Modin$^{{\href{mailto:klas.modin@chalmers.se}{\textrm{\Letter}}}*1}$
and Milo Viviani$^{{\href{mailto:milo.viviani@sns.it}{\textrm{\Letter}}}2}$}

\medskip

{\footnotesize
 \centerline{$^1$Mathematical Sciences, Chalmers and University of Gothenburg, Sweden}
} 

\medskip

{\footnotesize
 \centerline{$^2$Scuola Normale Superiore Pisa, Italy}
}

\bigskip


 \centerline{\textit{To the memory of Vladimir Zeitlin}}


\begin{abstract}
This paper gives a basic demonstration of matrix hydrodynamics; the field pioneered by V.~Zeitlin, where 2-D incompressible fluids are spatially discretized via quantization theory.
\end{abstract}


\section{Introduction}
In October 2024, we organized the mini-workshop ``Geometric and Stochastic Methods for Fluid Models'' at Kristineberg Station, one of the oldest marine biology stations in the world, situated in the western archipelago of Sweden.
Participants with mixed mathematical backgrounds united
in a zest for \emph{matrix hydrodynamics}, a field that began in the 1990s with the work of Vladimir Zeitlin, meteorologist and mathematician, who untimely passed away five months prior to the workshop.
This survey is dedicated to his legacy.

Zeitlin~\cite{Ze1991} took \emph{quantization theory}, developed for theoretical physics, and turned it into a numerical tool for spatial discretization of the 2-D Euler equations.
As a result, he found a finite-dimensional, isospectral matrix flow that captures the rich geometric structure of the 2-D Euler equations (Lie--Poisson structure, co-adjoint orbits, and Casimir functions, as surveyed below).
His first method was developed for the flat torus and became known as the ``sine-bracket approximation''.

Perhaps due to its unusual approach, Zeitlin's work quickly caught the eye of the numerics community.
McLachlan~\cite{Mc1993}, for example, found a suitable Lie--Poisson preserving time integration scheme based on splitting.
At large, however, numerical PDE analysts at the time concentrated on high-order methods, and in this respect, Zeitlin's discretization is mediocre (it has order $1/2$), so the interest faded.
Criticism also came from the fluid mechanics community, arguing that numerical methods for 2-D Euler \emph{should not} preserve enstrophy (one of the Casimirs) because enstrophy disperse into finer and finer scales, and this process cannot be resolved by any discretization.
Finally, quantization on the flat torus has a problem: the translational symmetry is spoiled, \edit{and spurious diffraction patterns appear}.
Zeitlin~\cite{Ze2004} later extended his method to the sphere, where quantization preserves the rotational symmetry, but at the time, 
the corresponding quantized Laplace--Beltrami operator could not be efficiently inverted, which blocked numerical simulations.

The objections above are well grounded.
Yet, in the last couple of years, Zeitlin's approach has seen a renaissance, with a renewed interest from both mathematicians and physicists and a growing community of devotees.
But what is behind this change in attitude?


One reason is that focus now lies on the direct matrix formulation, instead of the truncated Fourier coordinates in the sine-bracket formulation.
In turn, the matrix setting enables highly optimized numerical linear algebra routines, such as threaded BLAS\footnote{See \href{https://netlib.org/blas/}{netlib.org/blas/}.} algorithms for matrix-matrix multiplication.
The emphasis on matrix formulations also warrants ``matrix hydrodynamics'' as the collective name for quantization-based discretizations of hydrodynamic-like equations.

\edit{Another reason for the upswing is the gradual insight that efforts into high-order spatial discretizations are wasted when it comes to long-time simulations of the 2-D Euler equations: the local truncation error of high-order methods is governed by $C^k$- norms, and they typically grow faster than exponential (there are solutions where $C^1$ grows double exponential~\cite{KiSv2014}), so high order is of little use.}
Instead, it is better to focus on discretizations that capture the underlying geometric structure of the equations, like Zeitlin's method.

This point of view is strengthened by the physicists' statistical hydrodynamics approach.
Indeed, Onsager~\cite{On1949} worked out that statistical mechanics applied to the 2-D Euler equations predicts large-scale vortex condensation, such as observed in both experiments and real atmospheric flows (cf.~\cite{MaWa2006, BoEc2012, BoVe2012}).
Yet, statistical mechanics completely ignores the local dynamics; it is solely based on the geometric structure in phase space, namely the phase flow symplecticity and the conservation laws.
Matrix hydrodynamics is thus a bridge between statistical mechanics, based on phase space geometry, and traditional numerical PDE methods, based on locally tracking the dynamics.
A demonstrating example is the statistically predicted Kraichnan spectrum for the inverse energy cascade in 2-D hydrodynamics~\cite{Kr1967}, which is qualitatively captured with Zeitlin's method, almost independently of the matrix size, but which traditional methods fail to capture even at significantly higher resolutions~\cite{CiViLuMoGe2022}.

Finally, and most far-reaching, matrix hydrodynamics may itself give new mathematical tools for the long-time behavior of 2-D hydrodynamics.
An example is the \emph{canonical scale separation}~\cite{MoVi2022}, which
is based on Lie group theory enabled via the matrix formulation.
In the spirit of this example, matrix hydrodynamics is a device to distinguish finite-dimensional traits of the dynamics from exclusively infinite-dimensional traits.
For example, simulations with Zeitlin's model show that the inverse energy cascade is structurally stable under discretization as long as the underlying geometry is kept.
As another example, Lie--Poisson discretizations seem to qualitatively capture the predicted enstrophy distribution over length scales, albeit quantitatively exaggerated due to the finite-dimensionality---a new perspective on the above-mentioned criticism about conserving enstrophy (cf.~Section~\ref{sec:noise} below).

This survey is a hands-on introduction to matrix hydrodynamics.
It requires little more than linear algebra and vector calculus.
The focus is on Zeitlin's model on the sphere, including a brief but complete illustration of how to carry out simulations with the Python code \quflow .
The wider, ongoing applications of matrix hydrodynamics are outlined in Section~\ref{sec:topics}.
For a more in-depth introduction, including quantization theory and numerical convergence, we refer to another publication of ours~\cite{MoVi2024}.
\edit{The results in the paper are mostly known, except those in Section~\ref{sec:noise}.}

\section{Geometric hydrodynamics}

Let $M$ be a Riemannian manifold comprising the vessel of an incompressible, inviscid fluid.
Its motion is described by a time-dependent vector field $\boldsymbol{v}(\boldsymbol{x},t)$ on $M$ fulfilling Euler's~\cite{Eu1757} equations
\begin{equation}\label{eq:general_euler_equations}
    \dot{\boldsymbol{v}} + \nabla_{\boldsymbol{v}} \boldsymbol{v} = -\nabla p, \qquad \operatorname{div}\boldsymbol{v} = 0,
\end{equation}
where $\dot{\boldsymbol{v}}$ denotes the time derivative, 
the function $p$ is the pressure, and $\nabla_{\boldsymbol{v}}$ denotes the co-variant derivative in the direction of $\boldsymbol{v}$ (i.e., the natural generalization of the directional derivative, from functions to vector fields).

Geometric hydrodynamics began with Arnold's~\cite{Ar1966} discovery that Euler's equations~\eqref{eq:general_euler_equations} describe a geodesic motion on the group $\mathrm{SDiff}(M)$ of volume-preserving diffeomorphisms.
The key is to think of $\mathrm{SDiff}(M)$ as an infinite-dimensional Lie group.
For $\Phi\in \mathrm{SDiff}(M)$, the tangent space $T_\Phi\mathrm{SDiff}(M)$ then consists of maps of the form $\boldsymbol{u}\circ\Phi$ where $\boldsymbol{u}$ is a divergence-free vector field on $M$.
A Riemannian metric on $\mathrm{SDiff}(M)$ is obtained by assigning to tangent vectors $\boldsymbol{u}\circ\Phi$ and $\boldsymbol{v}\circ\Phi$ the right-invariant $L^2$ inner product
\begin{equation}\label{eq:euler_L2_metric}
    \langle \boldsymbol{u}\circ\Phi,\boldsymbol{v}\circ\Phi\rangle_{\Phi} = \int_M \boldsymbol{u}\cdot \boldsymbol{v} \, .
\end{equation}
Arnold realized that if $\boldsymbol{v}(x,t)$ is a solution to Euler's equations~\eqref{eq:general_euler_equations}, then the corresponding curve of diffeomorphisms, obtained by solving the non-autonomous ordinary differential equation
\begin{equation*}
    \dot \Phi = \boldsymbol{v}\circ\Phi ,\qquad (\;\text{i.e.,}\; \frac{\partial}{\partial t}\Phi(\boldsymbol{x}_0,t) = \boldsymbol{v}\big(\Phi(\boldsymbol{x}_0,t), t\big)\;),
\end{equation*}
is \emph{geodesic} (locally length minimizing) for the metric~\eqref{eq:euler_L2_metric}.
This Riemannian interpretation of incompressible hydrodynamics is more than a cute observation.
It is a powerful tool that yields insights to the dynamics.
For example, based on the geodesic formulation, Ebin and Marsden~\cite{EbMa1970} gave improved existence and uniqueness results.
Arnold himself made a calculation of sectional curvatures, showing that it is mostly negative, which in turn has implications for the stability of fluids and for the non-existence of accurate long-time weather predictions~\cite[appendix~2]{Ar1989}.
Many authors followed up on the curvature calculations, with further insights and results.

Although geometric hydrodynamics was invented for the Euler equations~\eqref{eq:general_euler_equations}, it is also, in its abstract form, applicable to a range of other equations (see Arnold and Khesin~\cite{ArKh2021} and references therein).
The general framework goes as follows.
If $G$ is a Lie group equipped with a right-invariant metric $\langle\cdot,\cdot\rangle_g$ such that the length of a tangent vector $\dot g = \xi g \in T_gG$ is
\begin{equation*}
    \langle \dot g,\dot g \rangle_g = \langle \dot g g^{-1}, \dot g g^{-1} \rangle_e = \langle \xi, \xi \rangle_e
\end{equation*}
then, in terms of the \emph{momentum variable} $\mu = \langle\xi,\cdot\rangle_e$, the geodesic equation is
\begin{equation}\label{eq:euler_arnold}
    \dot \mu + \mathrm{ad}^*_\xi \mu = 0,
\end{equation}
where $\mathrm{ad}^*_\xi\colon \mathfrak{g}^*\to\mathfrak{g}^*$ is the infinitesimal adjoint action of the Lie algebra $\mathfrak{g}=T_eG$ on its dual $\mathfrak{g}^*$.
The equation \eqref{eq:euler_arnold} is the \emph{Euler--Arnold equation} for the Lie algebra~$\mathfrak{g}$ equipped with the inner product $\langle \cdot,\cdot\rangle_e$.
\edit{It has a Hamiltonian structure as a Lie--Poisson system on the dual Lie algebra $\mathfrak{g}^*$ (\emph{cf.}~Marsden and Ratiu~\cite{MaRa1999}).}

Let us now derive these equations for the specific Lie algebra $\mathfrak{g} = C^\infty_0(S^2)$ of smooth, zero-mean functions on the sphere equipped with the Poisson bracket
\begin{equation}\label{eq:poisson_bracket_S2}
\{ \psi,\xi\}(\boldsymbol{x}) = \boldsymbol{x}\cdot(\nabla\psi\times \nabla\xi)
\end{equation}
for $\boldsymbol{x} \in S^2 \subset \mathbb{R}^3$.
The Hamiltonian vector field corresponding to $\psi$ is then given by the {skew gradient} $\nabla^\bot\psi = -\boldsymbol{x}\times\nabla\psi$, defined by $\nabla^\bot\psi\cdot\nabla\xi = \{\xi,\psi\}$.

As the inner product on the Lie algebra, we take the Dirichlet energy
\begin{equation}\label{eq:Dirichlet_energy}
    \langle \psi,\psi\rangle_e = \int_{S^2} \left|\nabla\psi\right|^2 .
\end{equation}
The corresponding momentum variable $\omega = \langle\psi,\cdot\rangle_e$ is naturally, via the $L^2$-pairing, represented as the function $\omega = -\Delta\psi$ for the Laplace-Beltrami operator $\Delta$ on $S^2$.
Indeed, from the divergence theorem
\begin{equation*}
    \langle\psi,\xi\rangle_e = \int_{S^2} \nabla\psi\cdot\nabla\xi = \int_{S^2}(-\operatorname{div}\nabla\psi)\xi =
    \int_{S^2}(\underbrace{-\Delta\psi}_{\omega})\xi \; .
\end{equation*}

It remains to work out the adjoint action operator $\mathrm{ad}_\psi^*$.
By definition,
\begin{equation*}
    \int_{S^2} (\mathrm{ad}_\psi^*\omega)\xi = \int_{S^2} \omega \{\psi,\xi \}.
\end{equation*}
Standard vector calculus identities then yield
\begin{multline*}
    \int_{S^2} \omega \{\psi,\xi \} = \int_{S^2} \omega \nabla\xi\cdot (\boldsymbol{x}\times\nabla\psi) = -\int_{S^2} \xi\operatorname{div}(\omega \boldsymbol{x}\times\nabla\psi) = \\  -\int_{S^2}\xi \nabla\omega\cdot (\boldsymbol{x}\times\nabla\psi) = -\int_{S^2}\xi \{\psi,\omega \}.
\end{multline*}
Thus, $\mathrm{ad}^*_\psi\omega = \{\omega,\psi\}$.\footnote{The calculation also shows that the $L^2$ inner product is \emph{bi-invariant}, namely $\langle \omega,\{\psi,\xi\}\rangle_{L^2} = \langle -\{\psi,\omega\},\xi\rangle_{L^2}$, which means that $L^2$ is canonical with respect to the Lie algebra structure. In particular, the geodesics generated by this inner product coincides with the group exponential.}
In summary, we have proved the following result.

\begin{proposition}
    The Euler--Arnold equation for the Lie algebra {\small $\mathfrak g = (C^\infty_0(S^2),\{\cdot,\cdot\})$} equipped with the Dirichlet energy inner product \eqref{eq:Dirichlet_energy} is
    \begin{equation}\label{eq:vorticity_equation}
        \dot\omega + \{\omega,\psi\} = 0, \qquad -\Delta\psi = \omega .
    \end{equation}
\end{proposition}

Next, we work out the connection to the Euler equations~\eqref{eq:general_euler_equations}.
The curl (or skew-divergence) of a vector field $\boldsymbol{v}$ on $S^2$ is given by
\begin{equation*}
    \operatorname{curl}\boldsymbol{v} = \operatorname{div}(\boldsymbol{x}\times \boldsymbol{v}).
\end{equation*}
Applying this operator to equation~\eqref{eq:general_euler_equations} yields
\begin{equation*}
    \operatorname{curl}\dot{\boldsymbol{v}} + \operatorname{curl}\nabla_{\boldsymbol{v}}\boldsymbol{v} = -\underbrace{\operatorname{curl}\nabla p}_{=0}.
\end{equation*}
Furthermore,
\begin{equation*}
    \operatorname{curl}\nabla_{\boldsymbol{v}}\boldsymbol{v} = \boldsymbol{v}\cdot \nabla\operatorname{curl}\boldsymbol{v}
\end{equation*}
which we leave as an exercise to prove.
Consequently,
\begin{equation*}
    \operatorname{curl}\dot{\boldsymbol{v}} + \boldsymbol{v}\cdot \nabla\operatorname{curl}\boldsymbol{v} = 0,
\end{equation*}
so $\operatorname{curl}\boldsymbol{v}$ is advected by the vector field $\boldsymbol{v}$.

Since the vector field $\boldsymbol{v}$ is divergence-free, and thus symplectic with respect to the area form, and since the first co-homology of the sphere is trivial (it lacks holes), we obtain that $\boldsymbol{v} = \nabla^\bot\psi$ for some Hamiltonian $\psi$, which is unique up to a constant (in particular, it is unique in $C^\infty_0(S^2)$).
Since $\operatorname{curl}\nabla^\bot\psi = \operatorname{div}(\boldsymbol{x}\times (-\boldsymbol{x}\times\nabla \psi)) = -\Delta \psi$ and since $\{\omega,\psi\} = \nabla^\bot\psi\cdot\nabla\omega$, we obtain that $\omega = \operatorname{curl}\boldsymbol{v}$ fulfills the Euler--Arnold equation~\eqref{eq:vorticity_equation}.
This is the \emph{vorticity formulation} of the 2-D Euler equations, and, in summary, we have the following result.

\begin{corollary}
    The Euler equations~\eqref{eq:general_euler_equations} on the domain $M=S^2$ are equivalent to the Euler--Arnold equation~\eqref{eq:vorticity_equation}.
    The vorticity function, corresponding to the momentum variable, is given by $\omega = \operatorname{curl}\boldsymbol{v}$, whereas the stream function $\psi$ is the Hamiltonian for the vector field, so that $\boldsymbol{v} = \nabla^\bot\psi$.
\end{corollary}

\section{The Euler--Zeitlin equations}

Zeitlin~\cite{Ze1991} had the brilliant idea to use quantization theory for spatial discretization of the vorticity form~\eqref{eq:vorticity_equation} of the 2-D Euler equations.
The aim of quantization theory is to find a mapping between smooth functions and Hermitian linear operators on a Hilbert space such that the Poisson bracket between functions corresponds, up to scaling by $\mathrm{i}/\hbar$, to the commutator between operators.
If the underlying domain is compact, then the Hilbert space can be finite-dimensional, so the operators are Hermitian matrices.
In the context of matrix hydrodynamics, we work with skew-Hermitian matrices $\mathfrak{u}(N)$ instead (multiplication by the imaginary unit $\mathrm{i}$ provides an isomorphism).

For the Poisson bracket~\eqref{eq:poisson_bracket_S2} on the sphere, we thus seek a linear mapping $$C^\infty(S^2) \ni \psi\mapsto T_N\psi \in \mathfrak{u}(N)$$ such that the image of the Poisson bracket $T_N\{\psi,\omega\}$ is approximated by the scaled commutator $\frac{1}{\hbar}[T_N\psi,T_N\omega]$ for $\hbar = 2/\sqrt{N^2-1}$.
Of course, this is trivially achieved by taking $T_N \equiv 0$, so we also require a consistency condition, typically $T_N 1 = -\mathrm{i} I$ (the constant function 1 is mapped to $-\mathrm{i}$ times the identity matrix) and $\lVert T^*_N T_N\psi-\psi \rVert_{L^2}\to 0$ as $N\to \infty$, where $T_N^*\colon \mathfrak{u}(N)\to C^\infty(S^2)$ denotes the $L^2$ adjoint of $T_N$.
Such a quantization was given by Berezin~\cite{Be1975} and by Hoppe~\cite{Ho1989}.
More generally, \emph{Berezin--Toeplitz quantization} is applicable to any  quantizable symplectic manifold (see the monograph by Le Floch~\cite{Le2018} for details).
Sharp convergence results were given by Charles and Polterovich~\cite{ChPo2018}.

For the sphere, quantization is best understood via representation theory.
Indeed, from a unitary $\mathfrak{so}(3)$ representation, generated by $S_1,S_2,S_3\in \mathfrak{u}(N)$, we obtain the \emph{Casimir operator}.
\begin{equation}\label{eq:hoppe_yau_laplacian}
    \Delta_N \colon \mathfrak{u}(N)\to \mathfrak{su}(N), \quad
    \Delta_N P = \sum_{\alpha=1}^3 [[P,S_\alpha],S_\alpha].
\end{equation}
As pointed out by Hoppe and Yau~\cite{HoYa1998}, this operator gives an approximation of the Laplace-Beltrami operator on $S^2$, with the same spectrum up to truncation at the maximum spherical harmonics wave number $\ell = N-1$.
For details on the representation point of view of quantization, see Modin and Viviani~\cite[sect.~3]{MoVi2024}.

We now have all the ingredients to approximate the 2-D Euler equations \eqref{eq:vorticity_equation} via quantization theory.
Indeed, the \emph{Euler--Zeitlin equation} is the isospectral flow of matrices on $\mathfrak{su}(N)$ given by
\begin{equation}\label{eq:euler-zeitlin}
    \dot W + \frac{1}{\hbar}[W, P] = 0, \qquad -\Delta_N P = W,
\end{equation}
where $\hbar = 2/\sqrt{N^2-1}$.
\edit{These equations constitute themselves an Euler--Arnold system on the Lie algebra $(\mathfrak{su}(N),\frac{1}{\hbar}[\cdot,\cdot])$ equipped with the quantized Dirichlet energy inner product
\begin{equation*}
    \langle P,P\rangle_I = \frac{4\pi}{N} \operatorname{tr}(P^\dagger (-\Delta_N)P),
\end{equation*}
analogous to the inner product~\eqref{eq:Dirichlet_energy}.
In particular, the system preserves the Lie--Poisson structure on $\mathfrak{su}(N)^*\simeq \mathfrak{su}(N)$ and the corresponding Hamiltonian function
\begin{equation*}
    H(W) = \frac{1}{2}\langle P,P\rangle_I = \frac{2\pi}{N}\operatorname{tr}(W^\dagger P) .
\end{equation*}
Under the adjoint quantization mapping $T_N^*$ as above, it gives an $\mathcal{O}(N^{-1})$ truncation of the Euler equations~\eqref{eq:vorticity_equation}.
}

The isospectral property of the flow~\eqref{eq:euler-zeitlin} reflects the advection property of the vorticity equation~\eqref{eq:vorticity_equation}.
It implies conservation of \emph{Casimir functions}
\[
    C_f^N(W) = \frac{4\pi}{N}\operatorname{tr}(f(\mathrm i W)),
\]
for any smooth function $f\colon \mathbb{R}\to\mathbb{R}$.
These converge to the corresponding continuous Casimirs $C_f(\omega) = \int_{S^2} f(\omega)$ as $N\to \infty$ \cite{MoVi2024}.
However, the key point of matrix hydrodynamics is \emph{not} the conservation of Casimirs.
Rather, the conservation laws stem from a deeper analog: just as the level sets of the continuous vorticity field $\omega$ are advected, the eigenvectors $\boldsymbol{e}_i$ of the vorticity matrix $W$ are ``advected" as
\begin{equation}\label{eq:advection_eigenvectors}
    \dot{\boldsymbol{e}}_i = \frac{1}{\hbar}P\boldsymbol{e}_i,
\end{equation}
while the corresponding eigenvalues $-\mathrm{i}\lambda_i$ of $W$ are preserved.\footnote{These equations allow for a ``transport formulation" of the Euler--Zeitlin equations, which can be used as a basis for time discretization \cite{CiMoPaVi2025_preprint}.}
This is readily seen since $W = \sum_{k=1}^N -\mathrm i\lambda_k \boldsymbol{e}_k\boldsymbol{e}_k^\dagger$ so the discrete advection equations \eqref{eq:advection_eigenvectors} yield
\begin{equation*}
    \dot W = \sum_{k=1}^N \frac{-\mathrm i \lambda_k}{\hbar} \big(\boldsymbol{e}_k (P \boldsymbol{e}_k)^\dagger + (P\boldsymbol{e}_k)\boldsymbol{e}_k^\dagger\big)
    = - \frac{1}{\hbar}[W,P].
\end{equation*}

See Figure~\ref{fig:advection} for an illustration of the correspondences
$$
\text{level sets} \leftrightarrow \text{eigenvectors}
\qquad \text{and}\qquad \text{vorticity values} \leftrightarrow
\text{eigenvalues}.
$$
These correspondences are at the heart of matrix hydrodynamics.
They enable numerical studies of the \emph{structural stability} of  2-D Euler equations~\eqref{eq:vorticity_equation},
namely, to study if long-time, qualitative properties of its solutions are stable under perturbations of the underlying infinite-dimensional Lie-Poisson structure (to the finite-dimensional Lie-Poisson structure for the Euler--Zeitlin equations).

\begin{figure}
    \includegraphics{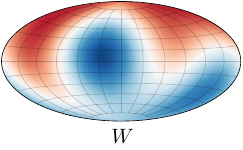}
    \includegraphics{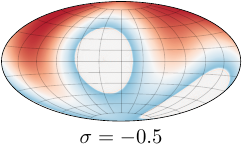}
    \includegraphics{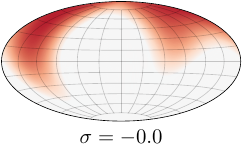}\\[2ex]
    \includegraphics{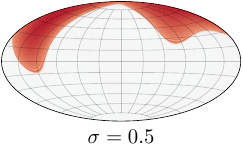}
    \includegraphics{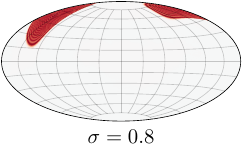}
    \includegraphics{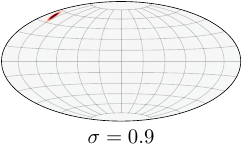}
    \caption{A skew-Hermitian matrix $W\in\mathfrak{su}(512)$, with eigenvalues $-\mathrm i\lambda_1,\ldots,-\mathrm i\lambda_{512}$ such that $-1\leq \lambda_m\leq 1$ and corresponding eigenvectors $\boldsymbol{e}_m$, is partially reconstructed by $\sum_{\mathrm i\lambda_m \geq \sigma} -\mathrm i\lambda_m \boldsymbol{e}_m\boldsymbol{e}_m^\dagger$ for $\sigma \in [-1,1]$.
    As seen in the figure, this decomposition corresponds to nullifying the level sets of the vorticity function with values below $\sigma$.
    The sphere is visualized via the area-preserving Hammer projection.
    }
    \label{fig:advection}
\end{figure}

\edit{

\subsection{Conservation of angular momentum}

In addition to Casimirs, the 2-D Euler equations on the sphere conserve angular momenta $$l_\alpha(\omega) \coloneqq \int_{S^2} x_\alpha \omega, \quad \alpha=1,2,3.$$
Geometrically, these three conservation laws arise from Noether's theorem due to the invariance of the Hamiltonian under the rotation group $\mathrm{SO}(3)$.
Zeitlin's model on $S^2$ inherits the same symmetry, since it stems from representations of $\mathrm{SO}(3)$, and therefore we obtain analog conservation laws.
As above, let $S_1,S_2,S_3$ denote the generators for the infinitesimal action of $\mathfrak{so}(3)$ on $\mathfrak{u}(N)$.

\begin{proposition}
    The Euler--Zeitlin equations \eqref{eq:euler-zeitlin} conserve the angular momenta given by
    \begin{equation*}
        L_\alpha(W) \coloneqq \frac{4\pi}{N}\operatorname{tr}(S_\alpha^\dagger W),\quad \alpha=1,2,3  .
    \end{equation*}
\end{proposition}

\begin{proof}
    We have
    \begin{multline*}
        \frac{d}{dt}L_\alpha(W) = -\frac{4\pi}{N} \operatorname{tr}(S_\alpha \dot W) = \frac{4\pi}{N\hbar}\operatorname{tr}(S_\alpha [\Delta_N P,P]) = \\ -\frac{1}{\hbar}\operatorname{tr}([\Delta_N P,S_\alpha] P) = -\frac{4\pi}{N\hbar}\operatorname{tr}([P,S_\alpha] \Delta_N P) = -\frac{4\pi}{N\hbar}\operatorname{tr}(S_\alpha [\Delta_N P,P]).
    \end{multline*}
    Thus, $\frac{d}{dt}L_\alpha(W) = - \frac{d}{dt}L_\alpha(W)$ so $\frac{d}{dt}L_\alpha(W) = 0$.
\end{proof}

}


\subsection{The isospectral midpoint method}
To simulate the Euler--Zeitlin equations~\eqref{eq:euler-zeitlin}, we need a suitable time integration method.
Ideally, to be useful for structural stability studies, the method should preserve the finite-dimensional Lie--Poisson structure of the Euler--Zeitlin equations.
In particular, it should preserve the isospectral property of the flow.\footnote{Notice, however, that preserving the Lie--Poisson structure is a much stronger condition than preserving isospectrality. Geometrically, isospectrality means that the motion remains on a single \emph{co-adjoint orbit}, whereas Lie--Poisson means that the flow is symplectic on that orbit.}
Such an integration scheme is given by the \emph{isospectral midpoint method}~\cite{MoVi2020c}, which is the mapping $W_{n}\mapsto W_{n+1}$ defined via an intermediate stage $\tilde{W}$ by
\begin{equation}\label{eq:isospectral_midpoint}
    \begin{aligned}
        &W_n = \tilde{W} - \frac{\varepsilon}{2}[\tilde{W}, \tilde{P}] - \frac{\varepsilon^2}{4}\tilde{P}\tilde{W}\tilde{P}, \qquad
        -\Delta_N \tilde{P} = \tilde{W}, \\
        &W_{n+1} = W_n - \varepsilon [\tilde{W},\tilde{P}],
    \end{aligned}
\end{equation}
where $\varepsilon = \delta t/\hbar$ for a chosen time step $\delta t$.
An efficient implementation of the scheme~\eqref{eq:isospectral_midpoint} is given by Cifani, Viviani, and Modin~\cite{CiViMo2023}.
\edit{It uses fixed-point iterations, resulting in $\mathcal{O}(N^3)$ complexity per time step; the quantized Laplacian $\Delta_N$ is inverted in $\mathcal{O}(N^2)$ operations, but matrix-matrix multiplication effectively requires $\mathcal{O}(N^3)$ operations.}

In addition to preserving the Lie--Poisson structure, the isospectral midpoint method is reversible with respect to $P\mapsto -P$.
Further, it is second-order accurate in time.\footnote{It is impossible to accurately track exact solutions for more than short time intervals, regardless of the order of the method.
However, higher-order time discretizations are still useful, for example, to more accurately conserve energy or in the study of stationary solutions~\cite{SiViLe2025_preprint}.
}

\section{A numerical illustration}\label{sec:illustration}

In this section, we demonstrate the Python package \quflow\footnote{Available at \href{https://github.com/klasmodin/quflow/}{github.com/klasmodin/quflow/}. It is installed via the \texttt{\footnotesize pip} package manager.} for simulating the 2-D Euler equations~\eqref{eq:vorticity_equation} via the Euler-Zeitlin equations~\eqref{eq:euler-zeitlin} discretized in time by the isospectral midpoint method~\eqref{eq:isospectral_midpoint}.
The first step is to load the package and specify the initial vorticity field by spherical harmonics coefficients.

\begin{lstlisting}
import numpy as np
import quflow as qf

# Initial conditions as random real spherical harmonics
# truncated at wave-number 'elmax'
elmax = 20
omega0 = np.random.randn((elmax+1)**2)

# Set vanishing mean
omega0[0] = 0.0

# Set vanishing total angular momentum
omega0[1:4] = 0.0

# Create corresponding initial matrix with size 'N'
N = 512
W0 = qf.shr2mat(omega0, N)

# Plot the initial vorticity field
qf.plot(omega0, colorbar=True)
\end{lstlisting}
The call to \code{qf.plot} displays the initial vorticity field, as seen in Figure~\ref{fig:initial_vorticity}.

\begin{figure}
    \includegraphics{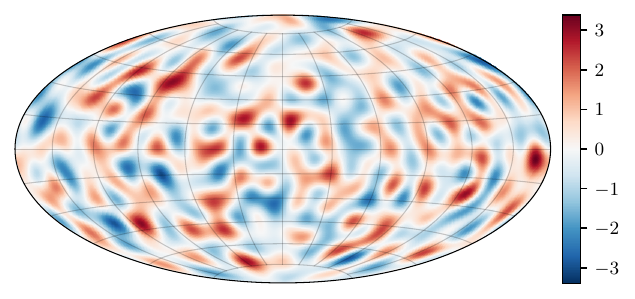}
    \caption{Initial vorticity field, generated as a truncated spherical harmonic series $\sum_{\ell=0}^{\ell_{\it max}}\sum_{m=-\ell}^\ell \omega_{\ell,m}Y_{\ell,m}$, where $\ell_{\it max}=20$ and the coefficients $\omega_{\ell,m}$ are normally distributed.
    } \label{fig:initial_vorticity}
\end{figure}

Next, the output data file is initialized, and the solver parameters are specified.

\begin{lstlisting}
# Create simulation data file
sim = qf.QuSimulation("quflow_sim.hdf5", state=W0, overwrite=True)

# Set simulation parameters
sim['dt'] = 0.2*qf.hbar(N)
sim['simtime'] = 100.0
sim['dt_out'] = 0.2
\end{lstlisting}
Notice that we have specified the time step length \code{dt} in units of $\hbar$.
This way, the time step scales correctly when we change the matrix size~$N$.

The final step is to run the simulation, i.e., to carry out time integration via the isospectral midpoint method~\eqref{eq:isospectral_midpoint}.
\begin{lstlisting}
# Run the simulation
qf.solve(sim)
\end{lstlisting}

Once the simulation is finished, the results are stored in the \code{hdf5} file format.
The default setting stores both function values on a grid in spherical coordinates (in the \code{hdf5}-field \code{'fun'}) and the corresponding $\mathfrak{su}(N)$ matrices (in the \code{hdf5}-field \code{'mat'}).
For example, the following code displays the last output and extracts the corresponding matrix.
\begin{lstlisting}
# Load simulation results
sim = qf.QuSimulation("quflow_sim.hdf5")

# Plot the last output
qf.plot(sim['fun', -1])

# Extract matrix of the last output
W = sim['mat', -1]
\end{lstlisting}

Snapshots of the simulated vorticity evolution are given in Figure~\ref{fig:vorticity_evolution}.
The initial vorticity field undergoes intricate mixing, where nearby vorticity regions of equal sign continue to merge until only a few ``vortex blobs'' remain.
Typically, the merging stops at 4 blobs if the total angular momentum is zero, corresponding to the integrable motion of 4 point vortices with vanishing angular momentum, and at 3 blobs for non-vanishing total angular momentum, corresponding to integrable motion of 3 point vortices.
More details on this mechanism, as well as further numerical experiments, are given by the authors in a prior publication~\cite{MoVi2020}.


\begin{figure}
    \includegraphics{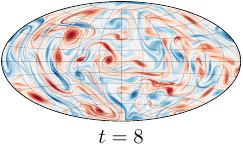}
    \includegraphics{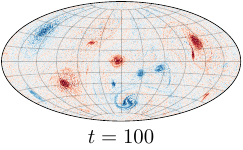}
    \includegraphics{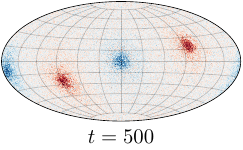}
    \caption{Evolution of the vorticity field for the \quflow\ simulation in Section~\ref{sec:illustration}, with initial data as in Figure~\ref{fig:initial_vorticity}.
    Vorticity regions of equal sign undergo mixing until four ``vortex blob condensates'' remain: two positive and two negative. After that, the large-scale motion stabilizes in quasi-periodic interaction between the blobs.}
    \label{fig:vorticity_evolution}
\end{figure}

\section{Conservation of enstrophy implies noise}\label{sec:noise}

Consider the expansion of vorticity in spherical harmonics
\begin{equation*}
    \omega = \sum_{\ell=1}^\infty\sum_{m=-\ell}^\ell \omega_{\ell,m}Y^{\ell,m},
\end{equation*}
where $\ell$ corresponds to the wave number frequency ($V^\ell = \mathrm{span}(Y^{\ell,-\ell},\ldots,Y^{\ell,\ell})$ is the eigenspace of the Laplace-Beltrami operator with eigenvalue $-\ell(\ell+1)$).
The long-time behavior of the 2-D Euler equations can be addressed by studying the statistical properties of the coefficients $\omega_{\ell,m}$.
Indeed, \edit{we expect a threshold wave-number $\ell^*$ such that}
for low wave numbers, $\ell < \ell^*$, we see a build-up of energy (cf.~\emph{inverse energy cascade}), whereas for high wave numbers, $\ell > \ell^*$, we see convergence toward a uniform distribution of the spherical harmonics coefficients (cf.~\emph{forward enstrophy cascade}).
In other words, for $\ell>\ell^*$, we expect, in average, that $\omega_{\ell,m}^2 \sim \varepsilon^2$ with an infinitesimal noise level $\varepsilon^2$ as $t\to\infty$.
Details of the underlying theoretical models, as well as numerical and experimental verification, are covered in the survey by Boffetta and Ecke~\cite{BoEc2012}.

As we have seen, the long-time simulations of the Euler-Zeitlin equations~\eqref{eq:euler-zeitlin} typically settle at a few interacting vortex blobs for the large-scale dynamics.
On small scales, however, there is a background ``noise'', moderate but significant in magnitude.
This noise is qualitatively correct, but quantitatively incorrect: its magnitude is exaggerated.

\begin{figure}
    \includegraphics{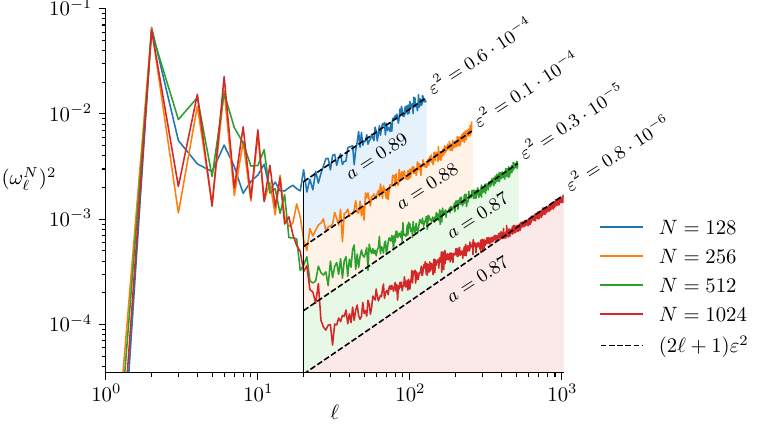}
    \caption{
        Visualization of the long-time spectral enstrophy components $(\omega_\ell^N)^2 = \sum_{m=-\ell}^\ell (\omega_{\ell,m}^N)^2$ for varying matrix sizes $N$.
        The initial data is the same as in Figure~\ref{fig:initial_vorticity}.
        On small enough scales, \edit{empirically found to be} $\ell > \ell^*\approx 20$, enstrophy gets uniformly distributed among the spherical harmonics coefficients $(\omega_{\ell,m}^N)^2$, such that $(\omega_\ell^N)^2 \approx (2\ell+1)\varepsilon^2$.
        For a smaller $N$, there is less available volume in phase space to distribute enstrophy, resulting in a larger background noise $\varepsilon^2$.
        The fact that the areas $a$ under the graphs are largely independent of~$N$ confirms this discussion.
    }
    \label{fig:spectrum}
\end{figure}

The magnification of the background noise is a consequence of conservation of enstrophy, i.e., conservation of the Casimir
\begin{equation*}
    C_2 = \int_{S^2} \omega^2 = \sum_{\ell=1}^\infty\sum_{m=-\ell}^\ell \omega_{\ell,m}^2 .
\end{equation*}
Indeed, for the Zeitlin model, we have
\begin{equation*}
    W = \sum_{\ell=1}^{N-1}\sum_{m=-\ell}^\ell \omega_{\ell,m}^N T_N^{\ell,m},
\end{equation*}
where $T_N^{\ell,m}$ are the eigenmatrices of the Hoppe--Yau Laplacian $\Delta_N$ directly corresponding to the spherical harmonics $Y^{\ell,m}$.
By introducing $(\omega_\ell^N)^2 = \sum_{m=-\ell}^\ell \omega_{\ell,m}^N$, and under the assumption $N>\ell^*$, we decompose the enstrophy as
\begin{equation*}
    C_2 = \sum_{\ell=1}^{\ell^*-1} (\omega_\ell^N)^2 + \underbrace{\sum_{\ell=\ell^*}^{N-1} (\omega_\ell^N)^2}_{a}.
\end{equation*}
In the long-time states, the large-scale blobs are accounted for in the first part of the sum.
It is natural to assume that this large-scale part is mostly unaffected by $N$, as long as $N$ is large enough in comparison to $\ell^*$ (e.g., at least twice as large).
Consequently, since enstrophy is conserved, $a$ is an adiabatic invariant largely independent of $N$.
If we furthermore assume that the Zeitlin model produces the qualitatively correct statistics, so that the small-scale states $(\omega_{\ell,m}^N)^2$ independently average to $\varepsilon^2$, we get that $(\omega_\ell^N)^2 \approx (2\ell+1)\varepsilon^2$.
These assumptions allow us to express $\varepsilon^2$ in terms of $a$ and $N$:
\begin{equation*}
    a = \sum_{\ell=\ell^*}^{N-1} (\omega_\ell^N)^2 \approx \varepsilon^2\sum_{\ell=\ell^*}^{N-1}(2\ell+1) \iff \varepsilon^2 \approx \frac{a}{\sum_{\ell=\ell^*}^{N-1}(2\ell+1)}.
\end{equation*}
In particular, we expect that the noise level $\varepsilon^2$ decreases as $\mathcal{O}(N^{-2})$ as $N\to\infty$, whereas $a$ remains largely independent of $N$.
Figure~\ref{fig:spectrum} confirms this behavior.

That enstrophy conservation in a numerical method for the 2-D Euler equations implies excessive background noise is an argument against conservation of enstrophy (and, more specifically, against conservation of co-adjoint orbits).
Indeed, it is common to introduce ``numerical viscosity'' to model the dispersion of entrophy into small scales (see, e.g., the study by Dritschel, Qi, and Marston~\cite{DrQiMa2015}).
However, the mechanism of enstrophy dispersion (forward cascade and mixing) in the 2-D Euler equations is distinct from the mechanism of viscosity or other types of damping.
Indeed, numerical viscosity tends to destroy the theoretically predicted properties of the long-time behavior, as such behavior often hinges on the advection of vorticity and thus on the conservation of co-adjoint orbits.
For example, the theoretically predicted Kraichnan spectrum for forced 2-D turbulence is captured in the Zeitlin model, chiefly independently of $N$, whereas traditional methods (with numerical viscosity) fail to capture it even at significantly higher resolutions~\cite{CiViLuMoGe2022}.
In our experience, for numerical studies of the long-time behavior, a viable strategy is to use Zeitlin-based discretizations and then carefully taking the background noise into account for the conclusions (the noise can be estimated via the model above).
Such simulations then capture the long-time behavior contingent on vorticity advection as long as that behavior is stable under amplified background noise.


\begin{figure}
    \includegraphics{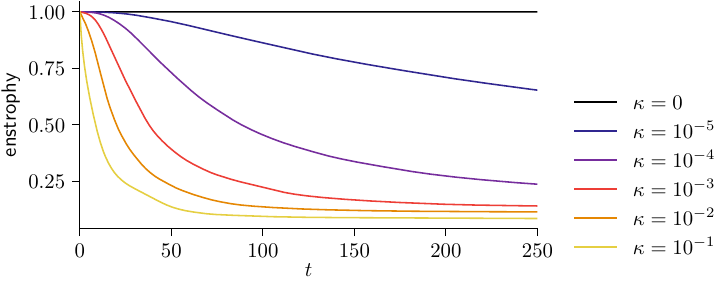}
    \caption{
        \edit{
        Evolution of enstrophy for canonical dissipation with different values of $\kappa$.
        The initial data is the same as in Figure~\ref{fig:initial_vorticity}, and the matrix size is $N=512$.
        The enstrophy is a convex Casimir and therefore decays, as predicted by Proposition~\ref{prop:dissipative_casimirs}.
        }
    }
    \label{fig:enstrophy_dissipation}
\end{figure}

\subsection{Canonical dissipation to reduce the noise}
There are situations where the noise induced by conservation of enstrophy is undesirable.
For example, we might want to compare the outcome of a structure-preserving, noisy simulation with one where enstrophy dissipates.
In such situations, one can introduce \emph{canonical dissipation}\footnote{This type of dissipation was introduced by Sadourny and Basdevant~\cite{SaBa1985} under the name ``anticipated vorticity method''.
It later led to a general method for dissipating Casimirs called ``selective decay'' \cite{GaHo2013}, which also fits into the setting of metriplectic dissipation (cf.~\cite{BrKrMaMo2025_preprint}).
However, both selective decay and metriplectic dissipation are general frameworks that rely on additional structure.
In contrast, the special case \eqref{eq:canonical dissipation} considered here is defined solely from the Hamiltonian and the Poisson bracket, and therefore we call it ``canonical dissipation''.} to the Euler--Zeitlin equations~\eqref{eq:euler-zeitlin}, namely
\begin{equation}\label{eq:canonical dissipation}
    \dot W + \frac{1}{\hbar}[W, P] = \frac{\kappa}{\hbar^2}[[W, P], P], \qquad -\Delta_N P = W,
\end{equation}
where $\kappa \geq 0$ is the canonical dissipation parameter (usually small).
The force is dissipative, and in this sense, similar to viscosity in the Navier--Stokes equations.
Unlike viscosity, however, it is non linear and preserves the energy.

\begin{proposition}\label{prop:energy_conservation_candis}
    The energy 
    \begin{equation*}
        H = \frac{2\pi}{N}\operatorname{tr}(W^\dagger P)
    \end{equation*}
    is conserved by the Euler--Zeitlin equations with canonical dissipation~\eqref{eq:canonical dissipation}.
\end{proposition}

\begin{proof}
Direct calculations give
\begin{multline*}
    \frac{d}{dt}\frac{2\pi}{N}\operatorname{tr}(PW) = \frac{4\pi}{N}\operatorname{tr}\big(P(\frac{1}{\hbar}[P,W] + \frac{\kappa}{\hbar^2}[P,[P,W]])\big)
    = \\  \frac{4\pi\kappa}{N\hbar^2}\operatorname{tr}(-\underbrace{[P,P]}_0[P,W]) = 0.
\end{multline*} \vspace{-8pt}
\end{proof}

On the other hand, the enstrophy decays with canonical dissipation, as seen in Figure~\ref{fig:enstrophy_dissipation}.
More generally, any convex Casimir decays.

\begin{proposition}\label{prop:dissipative_casimirs}
    Let $f\colon \mathbb{R}\to \mathbb{R}$ be a convex function.
    Then the corresponding Casimir
    \begin{equation*}
        C_f^N(W) = \frac{4\pi}{N}\operatorname{tr}(f(\mathrm i W))
    \end{equation*}
    decays along solutions of the Euler--Zeitlin equations with canonical dissipation~\eqref{eq:canonical dissipation}.
\end{proposition}

The proof relies on the following result.

\begin{lemma}\label{lem:eigendiff}
    Let $W,W'\in\mathfrak{u}(N)$ with $[W,W'] = 0$ and denote by $\mathrm iw_1,\ldots,\mathrm iw_N$ and $\mathrm iw'_1,\ldots,\mathrm i w'_N$ their eigenvalues, ordered with respect to a common eigenbasis $F\in U(N)$.
    Then, for any $X \in \mathfrak{u}(N)$,
    \begin{equation*}
        \operatorname{tr}([X,W]^\dagger[X,W']) = \sum_{m=-N}^N \sum_{i=1}^{n-|m|}|Y_{m\backslash i}|^2(w'_{i+|m|}-w'_{i})(w_{i+|m|}-w_{i}),
    \end{equation*}
    where $Y = F^\dagger X F$ and $Y_{m \backslash i}$ denotes the $i$:th element along the $m$:th diagonal of $Y$.
\end{lemma}

\begin{proof}
    Since $[W,W']=0$, we can find a common eigenbasis $F\in \mathrm{U}(n)$ that diagonalizes them.
    Then
    \begin{equation*}
        \mathrm{tr}([X,W]^\dagger [X,W']) 
        = \mathrm{tr}([Y,D_{W}]^\dagger [Y,D_{W'}])
    \end{equation*}
    where $Y=F^\dagger X F$, $D_{W} = \mathrm i\,\mathrm{diag}(w_1,\ldots,w_n)$, and $D_{W'} = \mathrm i\,\mathrm{diag}(w'_1,\ldots,w'_n)$.
    Via a direct calculation
    \begin{equation*}
        \mathrm{tr}([Y,D_{W}]^\dagger [Y,D_{W'}]) = \sum_{m=-N}^N \sum_{i=1}^{N-|m|}\lvert Y_{m\backslash i} \rvert^2 (p_{i+|m|} - p_{i})(w_{i+|m|} - w_{i}).
    \end{equation*}\vspace{-10pt}
\end{proof}

\begin{proof}[Proof of Proposition~\ref{prop:dissipative_casimirs}]
    Along the flow of \eqref{eq:canonical dissipation}, we have
    \begin{align*}
        \frac{d}{dt} C_f^N &= \frac{d}{dt}\frac{4\pi}{N}\operatorname{tr}(f(\mathrm i W)) =
        \frac{4\pi\mathrm i}{N}\operatorname{tr}\big(f'(\mathrm i W)(\frac{1}{\hbar}[P,W] + \frac{\kappa}{\hbar^2}[P,[P,W]])\big) = \\&\phantom{=}
        -\frac{4\pi\kappa\mathrm i}{N\hbar^2}\operatorname{tr}\big( [P, f'(\mathrm i W)][P,W] \big) = \\ &\phantom{=}
        - \frac{4\pi\kappa}{N\hbar^2}\sum_{m=-N}^N \sum_{i=1}^{n-|m|}|Y_{m\backslash i}|^2\big(f'(w_{i+|m|})-f'(w_{i})\big)(w_{i+|m|}-w_{i}),
    \end{align*}
    where the last equality follows from Lemma~\ref{lem:eigendiff} for $Y = F^\dagger P F$.
    Then, from the mean value theorem, we obtain
    \begin{multline*}
        - \frac{4\pi\kappa}{N\hbar^2}\sum_{m=-N}^N \sum_{i=1}^{n-|m|}|Y_{m\backslash i}|^2\big(f'(w_{i+|m|})-f'(w_{i})\big)(w_{i+|m|}-w_{i})
        = \\ - \frac{4\pi\kappa}{N\hbar^2}\sum_{m=-N}^N \sum_{i=1}^{n-|m|}|Y_{m\backslash i}|^2 f''(\xi_{im})(w_{i+|m|}-w_{i})^2
    \end{multline*}
    for $\xi_{im} \in [w_i,w_{i+|m|}]$.
    Since $f$ is convex, $f''\geq 0$, which proves the result.
\end{proof}

\edit{

\begin{figure}
    \includegraphics{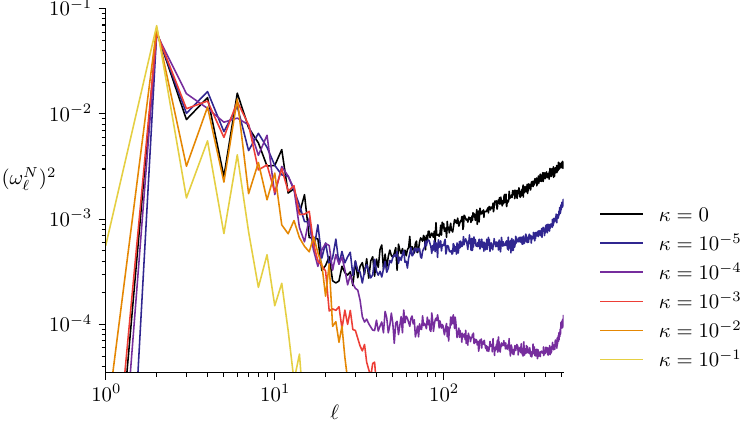}
    \caption{
        \edit{
        Visualization of the long-time spectral enstrophy components $(\omega_\ell^N)^2 = \sum_{m=-\ell}^\ell (\omega_{\ell,m}^N)^2$ in the presence of canonical dissipation of varying magnitude $\kappa$.
        The initial data is the same as in Figure~\ref{fig:initial_vorticity}, and the matrix size is $N=512$.
        The forward cascade of enstrophy is damped out.
        In order to qualitatively preserve the first part of the spectrum, the dissipation parameter must be small; for this simulation, at least $\kappa \leq 10^{-3}$.
        Notice that the canonical dissipation fails to conserve angular momenta; namely, the components with $\ell=1$ are not conserved.
        We cannot explain the growth toward the end of the spectrum for $0 < \kappa \leq 10^{-4}$; it seems to be an artifact of the quantization.
        }
    }
    \label{fig:spectrum_dissipation}
\end{figure}

\begin{remark}
    Canonical dissipation  can also be added to the continuous 2-D Euler equations, which gives the equations
    \begin{equation}\label{eq:euler_with_can_diss}
        \dot\omega + \{\omega,\psi\} = \kappa \{ \{\omega,\psi\},\psi\}, \qquad -\Delta\psi = \omega.
    \end{equation}
    The dimension of the dissipation parameter $\kappa$ is seconds $[\mathrm{s}]$.
    We can compare with kinematic viscosity $\nu$, which has the dimension of meter-squared per second $[\mathrm{m}^2 \mathrm{s}^{-1}]$.
    Notably, the canonical dissipation parameter $\kappa$ is independent of length scale~$[\mathrm{m}]$, contrary to viscosity, which dissipates small scales faster than large scales.
    Instead, how much the length scales are affected in canonical dissipation is not statically prescribed, as for kinetic viscosity, but is dynamically governed by the stream function~$\psi$.
    An example of how the long-time spectrum changes with $\kappa$ is given in Figure~\ref{fig:spectrum_dissipation}.
\end{remark}

\begin{remark}
    In the smooth setting, the initial vorticity $\omega_0$ is transported by the geodesic path of diffeomorphisms $\Phi_t(\boldsymbol{x})$ yielding the solution as $\omega_t = \omega_0 \circ \Phi_t^{-1}$.
    Geometrically, this means that the solution remains on the co-adjoint orbit of $\omega_0$, given by $\mathcal O(\omega_0) = \omega_0\circ \mathrm{SDiff}(S^2)$.
    The conservation of co-adjoint orbits prevents \emph{mixing} in the smooth topology. For example, if $\omega \in \mathcal{O}(\omega_0)$, then $\omega_0$ and $\omega$ must have the same number of critical points, so, in the smooth topology, it cannot happen that two vorticity blobs merge into one.
    Nevertheless, such mixing is in essence what we observe.
    To relax the constraints induced by the smooth topology, Shnirelman~\cite{Sh1993} suggested the $L^\infty$ weak-$*$ topology.
    The closure of the co-adjoint orbit $\mathcal O(\omega_0)$ in this weak topology allows mixing.
    Shnirelman further suggested that long-time states are given by maximally mixed states, which he called \emph{minimal flows}.
    The energy functional is continuous in the weak-$*$ topology, but the Casimirs are not.
    Indeed, mixing is characterized by decay of convex Casimir.
    Dolce and Drivas~\cite{DoDr2022a} showed that on the weak-$*$ closure of a co-adjoint orbit, intersected with a fixed energy surface, the minimization of convex Casmirs leads to minimal flows.
    Thus, we expect solutions to the 2-D Euler equations with canonical dissipation~\eqref{eq:euler_with_can_diss} to converge to minimal flows as $t\to\infty$.
\end{remark}

\begin{remark}
    The canonical dissipation term $\{\{\omega,\psi\},\psi\}$ corresponds to diffusion along the level sets of the stream function $\psi$.
    Indeed, it can be written $X_\psi[X_\psi[\omega]]$, where $X_\psi = \nabla^{\bot}\psi$ is the Hamiltonian vector field for $\psi$.
    For a fixed $\psi$, the equation
    \begin{equation*}
        \dot\omega = \{\{\omega,\psi\},\psi\}
    \end{equation*}
    therefore describes one-dimensional heat flows along the streamlines of~$\psi$, where the strength of the dissipation is determined by the magnitude of $X_\psi$ (or, equivalently, the magnitude of the gradient )
    In the limit $t\to\infty$, we thus obtain an averaging of $\omega_0$ along the streamlines.
    A direct construction of this averaging is the \emph{canonical splitting}~\cite{MoVi2022}, which decomposes the vorticity as $\omega = \omega_s + \omega_r$, where $\omega_s$ is given by averages along the streamlines, and $\omega_r$ is the residual.
    From a Lie theoretic point of view, the canonical splitting component $\omega_s$ is obtained as the projection of the vorticity function $\omega$ on the stabilizer of the stream function $\psi$.
    The components $\omega_s$ and $\omega_r$ are $L^2$ orthogonal by construction, but we also have that $\omega_r$ is $H^{-1}$ (energy norm) orthogonal to $\omega$, which reflects the energy conservation property of canonical dissipation (cf.~Proposition~\ref{prop:energy_conservation_candis}).
    In summary, canonical dissipation is an infinitesimal generator for the canonical splitting.
\end{remark}
}

\section{Near conservation of energy}

The isospectral midpoint method~\eqref{eq:isospectral_midpoint} exactly conserves Casimirs for the Euler--Zeitlin equations~\eqref{eq:euler-zeitlin}.
But it fails to exactly conserve energy.
One can view this as a consequence of the method being symplectic on the co-adjoint orbits: a method exactly conserving energy and symplecticity for a non-integrable system implies exact tracking of trajectories up to time-reparameterization \cite{ZhMa1988}.
Yet, since the co-adjoint orbits for $\mathfrak{su}(N)$ are compact for a fixed $N$, we can expect that energy is ``nearly conserved'' in the sense of backward error analysis (BEA) for symplectic integrators (cf.~Hairer, Lubich, and Wanner~\cite{HaLuWa2006} and references therein).
But there is a problem with this argument: we are interested in the behavior for large~$N$, whereas there is no control on how the estimates in BEA depend on $N$.
Therefore, BEA, as developed for finite-dimensional Hamiltonian system, cannot rigorously be used to address near conservation of energy in simulations based on the Euler--Zeitlin equations.
Numerical experiments nevertheless suggest that energy is nearly preserved independently of $N$, as illustrated in Figure~\ref{fig:energy}.

\begin{openproblem}
    Develop backward error analysis (for the isospectral midpoint method applied to the Euler--Zeitlin equations) that rigorously explains the numerically observed near conservation of energy.
\end{openproblem}

\begin{figure}
    \includegraphics{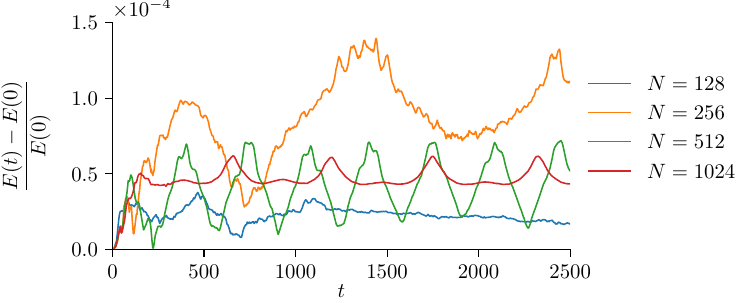}
    \caption{Evolution of the relative energy error for varying matrix sizes $N$. The initial data is the same as in Figure~\ref{fig:initial_vorticity}.
    Notice that the energy is ``nearly conserved'' in the sense compatible with the backward error analysis for symplectic integrators.
    There is no indication that the errors grow with $N$.
    We currently lack a rigorous analysis of this behavior, since backward error analysis, as developed for finite-dimensional Hamiltonian systems, prevents control over $N$.
    }
    \label{fig:energy}
\end{figure}

\section{Further topics}\label{sec:topics}
This survey focuses on the Euler--Zeitlin equations~\eqref{eq:euler-zeitlin} to approximate the 2-D Euler equations on the sphere.
But this setting is only the beginning of matrix hydrodynamics: the approach is viable for a great number of equations in mathematical physics.
The general principle is to use quantization theory to replace functions by matrices and the Poisson bracket by the (scaled) commutator.
Below we exemplify recent developments.

\subsection{Geophysical fluid dynamics}
Two-dimensional models naturally appear in geophysical fluid dynamics, via the quasi-geostrophic equation and its variants (cf.~\cite{Ze2018}).
A matrix hydrodynamics version of the global quasi-geostrophic equation was given by Franken, Caliaro, Cifani, and Geurts~\cite{FrCaCiGe2024}, where the authors provided simulations that qualitatively captured the turbulent band structures observed in the atmosphere of Jupiter.
Based on this work, further analysis of the critical latitude for jet formation was provided by the same authors~\cite{FrLuEpGu2024b}.

The underlying Riemannian geometric structure of the global quasi-geostrophic equations was specified by Modin and Suri~\cite{MoSu2025}.
This structure enables stability analysis based on sectional curvature calculations.
In particular, the study shows that the \emph{Lamb parameter} in the global quasi-geostrophic model has a stabilizing effect on the dynamics. 

The matrix hydrodynamics approach has also been extended to multi-layer quasi-geostrophic models on the sphere~\cite{FrLuEpGu2025}.
The resulting method enables long-term simulations without additional regularization, and the study highlights the importance of structure-preserving techniques for understanding large-scale geophysical flows.

Roop and Ephrati~\cite{RoEp2025} also derived a global model of thermal quasi-geostrophy on the sphere, based on the thermal rotating shallow water equations, and gave a corresponding matrix hydrodynamics discretization.
Simulations with their method reveal the formation of vorticity and buoyancy fronts, and large-scale structures in the buoyancy dynamics induced by the buoyancy-bathymetry interaction.

\subsection{Reduced magnetohydrodynamics}
Reduced magnetohydrodynamics constitute a two-dimensional system of PDEs and a simple model for both astrophysical and laboratory plasmas.
A matrix hydrodynamics model for this system, in the case of the flat torus, was given by Zeitlin~\cite{Ze2005}.
Contrary to the Euler--Zeitlin equations, the matrix flow is only partially isospectral.
An analogous model on the sphere, together with an extension of the spherical midpoint method to the underlying semi-direct Lie algebra, was given by Modin and Roop~\cite{MoRo2025}.

\subsection{Stochastic hydrodynamics}
The Euler--Zeitlin equations and the isospectral midpoint method can be extended to stochastic versions, via the addition of \emph{transport noise}, as developed by Ephrati, Jansson, Lang, and Luesink~\cite{EpJaLaLu2024}.
The authors proved almost sure preservation of Casimir functions and coadjoint orbits under  numerical flow and provided strong and weak convergence rates of the proposed method.

\subsection{Model reduction techniques}
The isospectral matrix flow formulations in matrix hydrodynamics allow for model reduction, based, for example, on the matrix eigenspaces.
Such considerations were given by Ephrati, Cifani, Viviani, and Geurts~\cite{CiEpVi2023,EpCiViGe2023} and by Pagliantini~\cite{Pa2025}.
This approach has the potential to significantly reduce the complexity of the Euler--Zeitlin equations and the isospectral midpoint methods, from $\mathcal{O}(N^3)$ to $\mathcal{O}(N^2)$, thus on par with the complexity of conventional methods such as finite elements, but retaining the structural benefits of matrix hydrodynamics.

\subsection{Axisymmetric three-dimensional Euler equations}
All the developments listed so far have been on two-dimensional domains (typically the sphere).
Indeed, the matrix hydrodynamics approach fails to directly extend to three-dimensional domains.
Essentially, this limitation stems from quantization, which relies on the symplectic structure of two-dimensional, compact manifolds, and the fact that incompressibility then implies that the fluid velocity field is Hamiltonian.
Nevertheless, the \emph{axisymmetric} 3-D Euler equations on the 3-sphere admits a matrix hydrodynamics discretization, as developed by Modin and Preston~\cite{MoPr2025}.
Numerical simulations based on this method show a faster-than-exponential growth of the maximum vorticity.
The authors also used the finite-dimensional model to study Riemannian curvature and Jacobi equations in the context of axisymmetric 3-D Euler.


\section*{Acknowledgments}
This work was supported by the Swedish Research Council (grant number 2022-03453), the Knut and Alice Wallenberg Foundation (grant numbers WAF2019.0201), and the Göran Gustafsson Foundation for Research in Natural Sciences and Medicine.
The computations were enabled through resources provided by Chalmers e-Commons at Chalmers and by the National Academic Infrastructure for Supercomputing in Sweden (NAISS), partially funded by the Swedish Research Council through grant no.~2022-06725.
The authors also thank SVeFUM for funding the workshop ``Geometric and Stochastic Methods for Fluid Models'', which eventually led to the special issue at hand.
Finally, we would like to acknowledge Sagy Ephrati for valuable discussions about canonical dissipation.










\bibliographystyle{AIMS}
\bibliography{bibliography}

\medskip
Received xxxx 20xx; revised xxxx 20xx; early access xxxx 20xx.
\medskip

\end{document}